\documentclass[11pt]{amsart}
\usepackage{amsmath, amsthm, amssymb}
\usepackage{epsfig}
\usepackage{rawfonts}
\usepackage{enumerate}
\usepackage{graphics}
\usepackage{tikz}
\usepackage{booktabs}

\usepackage{xspace}
     \usepackage{graphicx}

\newcommand{\agrave}{\`a\ }

\newcommand{\be}{\mathbf{e}}
\newcommand{\Axes}{\textup{Axes}}
\newcommand{\N}{\mathbb{N}}
\newcommand{\kthick}[1]{\textup{T}_{#1}}
\newcommand{\genwilf}{the Generalized Wilf Conjecture}

\theoremstyle{plain}

\theoremstyle{theorem}

\newtheorem{defn}{Definition}[section]
\newtheorem{prop}[defn]{Proposition}
\newtheorem{thm}[defn]{Theorem}
\newtheorem{lemma}[defn]{Lemma}
\newtheorem{conj}[defn]{Conjecture}
\newtheorem{coro}[defn]{Corollary}
\newtheorem{exa}[defn]{Example}
\newtheorem{rmk}[defn]{Remark}

\newtheorem*{thm2}{Generalized Wilf Conjecture}
\theoremstyle{remark}
\newtheorem{rem}[defn]{Remark}


\thanks{
The research that led to the present paper was partially supported by a grant of the  group GNSAGA of INdAM.  DiPasquale, Peterson, and Flores would like  to thank the University of Messina for its financial support. The research of DiPasquale, Peterson, and Flores was also supported in part by NSF grant DMS-1830676.
}

\begin{document}

\title[A generalization of Wilf's conjecture for GNS ]{A generalization of Wilf's conjecture for Generalized Numerical Semigroups}

\author[C. Cisto, M. DiPasquale, G. Failla, Z. Flores, C. Peterson R. Utano]{Carmelo Cisto, Michael DiPasquale, Gioia Failla, Zachary Flores, Chris Peterson, Rosanna Utano}

\address{Universit\agrave di Messina, Dipartimento di Matematica e Informatica\\
Viale Ferdinando Stagno D'Alcontres 31\\
98166 Messina, Italy}
\email{carmelo.cisto@unime.it}

\address{Colorado State University\\
Department of mathematics, Fort Collins, CO 80523 USA}
\email{michael.dipasquale@colostate.edu}

\address{Universit\agrave Mediterranea di Reggio Calabria, DIIES\\
Via Graziella, Feo di Vito,  Reggio Calabria, Italy}
\email{gioia.failla@unirc.it}

\address{Colorado State University\\
Department of mathematics, Fort Collins, CO 80523 USA}
\email{flores@math.colostate.edu}

\address{Colorado State University\\
Department of mathematics, Fort Collins, CO 80523 USA}
\email{peterson@math.colostate.edu}

\address{Universit\agrave di Messina, Dipartimento di Matematica e Informatica\\
Viale Ferdinando Stagno D'Alcontres 31\\
98166 Messina, Italy}
\email{rosanna.utano@unime.it}

\begin{abstract}

A numerical semigroup is a submonoid of $\mathbb N$ with finite complement in $\mathbb N$.  A generalized numerical semigroup is a submonoid of $\mathbb{N}^{d}$ with finite complement in $\mathbb{N}^{d}$.  In the context of numerical semigroups, Wilf's conjecture is a long standing open problem whose study has led to new mathematics and new ways of thinking about monoids. A natural extension of Wilf's conjecture, to the class of $\mathcal C$-semigroups, was proposed by Garc\'ia-Garc\'ia, Mar\'in-Arag\'on, and Vigneron-Tenorio. In this paper, we propose a different generalization of Wilf's conjecture, to the setting of generalized numerical semigroups, and prove the conjecture for several large families including the irreducible, symmetric, and monomial case.  We also discuss the relationship of our conjecture to the extension proposed by Garc\'ia-Garc\'ia, Mar\'in-Arag\'on, and Vigneron-Tenorio.
\end{abstract}


\maketitle

\section*{Introduction}
Let $S$ be a submonoid of $\mathbb{N}$. The \emph{hole set} of $S$ is defined as $H(S)=\mathbb{N}\setminus S$. If $H(S)$ is a finite set then $S$ is called a \emph{numerical semigroup}. Every numerical semigroup $S$ admits a unique, minimal, finite set of generators $G(S)$. Thus, there is a unique finite set $G(S)$ such that every element of $S$ is an $\mathbb N$-linear combination of elements in $G(S)$ while no proper subset of $G(S)$ has the same property. Well-known invariants of a numerical semigroup $S$ are $e(S) = |G(S)|$, $F(S)=\max\{k \mid k\in H(S)\}$, and $n(S)=|\{s\in S\mid s<F(S)\}|$. For relationships between these and other invariants of numerical semigroups see \cite{Garcia-Book}. An intriguing matter of study concerns the conjectured inequality $e(S)n(S)\geq F(S)+1$. This inequality is known as \emph{Wilf's conjecture} because of its first appearance in \cite{W.}. While Wilf's conjecture has been proved for several classes of numerical semigroups (see for instance \cite{D.M., Eliahou, K.N., Sammartano}), it is still open in general.  The survey~\cite{wilfSurvey} is a good reference for the state of the art on this long standing open problem.

A monoid $S\subseteq \mathbb{N}^{d}$ is called a \emph{generalized numerical semigroup} (GNS) if the hole set $H(S)=\mathbb{N}^{d}\setminus S$ is finite. As in the case where $d=1$, a generalized numerical semigroup, $S$, has a unique, minimal, finite set of generators, $G(S)$ and we let $e(S)=|G(S)|$. Additional properties and features of generalized numerical semigroups are provided in \cite{Analele,Work1, Prof}. A problem posed in \cite{Prof} was to formulate extensions of Wilf's conjecture to the setting of generalized numerical semigroups. A first possible extension was given in \cite{G.G.2016} for a larger class of semigroups called $\mathcal{C}$-semigroups. Their extension is quite natural and has the additional feature of depending on a monomial order.  We propose an alternate generalization of Wilf's conjecture to to the setting of generalized numerical semigroups.  Suppose $\textbf{x}, \textbf{y} \in \mathbb N^d$ with $\textbf{x}= (x^{(1)},\ldots,x^{(d)})$ and $\textbf{y}=(y^{(1)},\ldots,y^{(d)})$. There is a natural partial order, $\leq$ on $\mathbb N^d$, by setting $\textbf{x}\leq \textbf{y}$ if and only if $x^{(i)}\leq y^{(i)}$ for all $i=1,\ldots,d$.   Using this partial order, define $n(S)=|\{\textbf{x}\in S\ |\ \textbf{x}\leq \textbf{h}\ {\rm for\  some}\  \textbf{h}\in H(S)\}|$ and   $c(S)=|\{\textbf{x}\in \mathbb N^d\ |\ \textbf{x}\leq \textbf{h} \ {\rm for\  some}\   \textbf{h}\in H(S)\}|$. With this notation in place, we propose the following generalization of Wilf's conjecture to the setting of generalized numerical semigroups: 
\begin{thm2} If $S\subset \mathbb N^d$ is a GNS then $e(S)\hskip 1pt n(S)\geq d\hskip 1pt c(S)$.  \end{thm2}

\noindent This paper is concerned with motivating the above conjecture, proving it for several large classes of generalized numerical semigroups, and contrasting it with the extension of Wilf's conjecture proposed in \cite{G.G.2016}.  Throughout this paper we will refer to the above conjecture as \textit{the Generalized Wilf Conjecture}.


In Section 1 we recall the most important properties about {\it irreducible} generalized numerical semigroups, a class studied in \cite{Work1}.  Furthermore, we explain why \genwilf{}  can be considered as a generalization of Wilf's conjecture.  In Section 2 we introduce an operation called \textit{thickening} and in Section 3 we use this operation to prove that irreducible generalized numerical semigroups satisfy the proposed conjecture.  In Section 4 we consider the class of \emph{monomial semigroups}, i.e. semigroups satisfying the statistic $n(S)=1$, and we show that elements in this class also satisfy the Generalized Wilf Conjecture.  In Section 5 we compare our proposed extension of Wilf's conjecture with the one given in \cite{G.G.2016}. The last two sections are devoted to providing additional computational evidence for \genwilf{} and concluding remarks.\\

\section{Frobenius, irreducible, and symmetric semigroups and \genwilf{}}


In this section, we discuss \textit{Frobenius} and \textit{irreducible} generalized numerical semigroups as introduced in~\cite{Work1}.  We motivate our statement of \genwilf{} by considering how Wilf's conjecture can be extended to Frobenius semigroups, and we prove \genwilf{} for a certain class of Frobenius semigroups.  We also fix the basic notation and vocabulary that will be used throughout the paper.  

Throughout the paper, $S$ refers to a generalized numerical semigroup.  The set of {\it pseudo-Frobenius elements} of $S$ is $PF(S)=\{\textbf{h}\in H(S)\mid \textbf{h}+S \subset S\}$ while the set of {\it special gaps} is $EH(S) =\{\textbf{h}\in PF(S)\mid 2\textbf{h}\in S\}$. $S$ is called {\it irreducible} if it is not possible to express $S$ as the intersection of two larger generalized numerical semigroups. We have the following theorems.

\begin{thm}[\cite{Work1}, Theorem 2.9] Let $S\subseteq \mathbb{N}^{d}$ be a GNS. Then the following statements are equivalent:
\begin{enumerate}
\item $|PF(S)|=1$.
\item $PF(S)=\{\textbf{f}\}$ and $\textbf{f}$ has at least one odd component.
\item There exists an $\textbf{f}\in H(S)$ such that $\textbf{f}-\textbf{h}\in S$ for all $\textbf{h}\in H(S)$.
\end{enumerate}
\label{sym}\end{thm}

\begin{thm}[\cite{Work1}, Theorem 2.10] Let $S\subseteq \mathbb{N}^{d}$ be a GNS. Then the following statements are equivalent:
\begin{enumerate}
\item $PF(S)=\{\textbf{f}, \frac{\textbf{f}}{2}\}$.
\item There exists an $\textbf{f}\in H(S)$, with even components, such that $\textbf{f}-\textbf{h}\in S$ for all $\textbf{h}\in H(S) \smallsetminus \{\frac{\textbf{f}}{2}\}$.
\end{enumerate}
\label{p-sym}\end{thm}

\begin{defn} \rm $S$ is \emph{symmetric} if it satisfies the conditions of Theorem~\ref{sym}. $S$ is \emph{pseudo-symmetric} if it satisfies the conditions of Theorem~\ref{p-sym}. 
 In both cases $EH(S)=\{\textbf{f}\}$.  \end{defn}

\begin{defn} \rm If there exists a unique maximal element $\textbf{f}\in H(S)$, with respect to the natural partial order in $\mathbb{N}^d$, then we say that $S$ is {\it Frobenius} with  \emph{Frobenius element} $\textbf{f}$. We say that $(S,\textbf{f})$ is a {\it Frobenius GNS}. \end{defn}

\begin{thm}[\cite{Work1}]Let $S\subseteq \mathbb{N}^{d}$.
	\begin{enumerate}
	\item [1)] $S$ is irreducible if and only if $|EH(S)|=1$.
		\item [2)]  $S$ is symmetric if and only if there exists $\textbf{f}\in H(S)$ with $2|H(S)|=(f^{(1)}+1)(f^{(2)}+1)\cdots (f^{(d)}+1)$. \label{rel}
		\item [3)] $S$ is pseudo-symmetric if and only if there exists $\textbf{f}\in H(S)$ with $2|H(S)|-1=(f^{(1)}+1)(f^{(2)}+1)\cdots (f^{(d)}+1)$. \label{p-rel}
	\end{enumerate}
In each case, $(S,\textbf{f})$ is a Frobenius GNS  where in 1), we assume $EH(S) =\{\textbf{f}\}$.
\label{irr_formula}\end{thm}

Recall that $S\subset \mathbb N^d$ has a unique finite set of minimal generators (see \cite{Analele}).
We will use the following notation throughout the paper:
\begin{defn}
Given a generalized numerical semigroup $S\subset \mathbb N^d$ we define
\begin{itemize}
\item $G(S)=$ the set of minimal generators of $S$
\item $H(S)=\N^d\setminus S$
\item $C(\textbf{h})=\{\textbf{x}\in \mathbb N^d \mid \textbf{x}\leq \textbf{h}\}$
\item $C(S)=\{\textbf{x}\in \mathbb N^d \mid \textbf{x}\leq \textbf{h} \ {\rm for\  some}\  \textbf{h} \in H(S)\}$
\item $N(\textbf{h})=\{\textbf{x}\in  S \mid \textbf{x}\leq \textbf{h}\}$
\item $N(S)=\{\textbf{x}\in  S \mid \textbf{x}\leq \textbf{h} \ {\rm for\  some}\  \textbf{h} \in H(S)\}$
\item $H(\textbf{h})=\{\textbf{x}\in H(S) \mid \textbf{x}\leq \textbf{h}\}$
\item  $n(S)=|N(S)|, c(S)=|C(S)|$, $g(S)=|H(S)|$, and $e(S)=|G(S)|$
\end{itemize}
\end{defn}

 Consider the map:
$$\Psi_{\textbf{h}}:N(\textbf{h})\rightarrow H(\textbf{h}) \ \ {\rm defined \ by} \ \ \Psi_{\textbf{h}}(\textbf{s}) = \textbf{h}-\textbf{s}.$$
It is elementary to show that $\Psi_{\textbf{h}}$ is well defined and injective. As a consequence, $|N(\textbf{h})|\leq |H(\textbf{h})|\leq |H(S)|$.


\begin{prop} Let $S\subseteq \mathbb{N}^{d}$ be a symmetric GNS with Frobenius element $\textbf{f}$. Then $e(S)n(S)\geq d(f^{(1)}+1)\cdots(f^{(d)}+1)$. \label{w-sym}\end{prop}
\begin{proof}
Since $S$ is Frobenius, we have  $|H(\textbf{f})|=|H(S)|, |N(\textbf{f})|=|N(S)|$, and $|C(\textbf{f})|=|C(S)|$. Since $S$ is symmetric, we have $2|H(S)|=(f^{(1)}+1)(f^{(2)}+1)\cdots (f^{(d)}+1)=|C(\textbf{f})|=|C(S)|=|N(S)|+|H(S)|$. This implies that $n(S)=|N(S)|=|H(S)|$. Since $e(S)\ge 2d$ by~\cite[Theorem~11]{G.G.2016}, we have $e(S)n(S)\geq 2d|H(S)| = d(f^{(1)}+1)\cdots(f^{(d)}+1)$ (by Theorem~\ref{rel}). 
\end{proof}

Let $S$ be a symmetric numerical semigroup with Frobenius number $F(S)$, then by Proposition~\ref{w-sym} with $d=1$, $S$ satisfies $e(S)n(S)\geq F(S)+1$. This inequality is known as Wilf's conjecture and has been shown to be satisfied by several classes of numerical semigroups.  In general, the conjecture is wide open and is one of the long standing open problems in the study of numerical semigroups. The proposition above suggests a straightforward generalization for the Wilf conjecture for Frobenius generalized numerical semigroups.

\begin{conj} (The Generalized Wilf Conjecture for Frobenius GNS) Let $(S,\textbf{f})$ be a Frobenius GNS in $\mathbb{N}^{d}$. Then $e(S)n(S)\geq d(f^{(1)}+1)\cdots(f^{(d)}+1)$. \label{wilf-frof}\end{conj}

  Observe that, if $S\subseteq \mathbb{N}^{d}$ is a GNS and $\textbf{h}\in H(S)$ then $$|C(\textbf{h})|=|N(\textbf{h})|+ |H(\textbf{h})|=(h^{(1)}+1)(h^{(2)}+1)\cdots (h^{(d)}+1).$$

The key idea in the previous conjecture is to substitute the value $F(S)+1$, for numerical semigroups, with the cardinality of the set $C(\textbf{f})$ in the case of Frobenius generalized numerical semigroups $(S,\textbf{f})$, for which there exists a unique Frobenius element. However there are more general situations and $F(S)+1$ may be replaced in a different way.
Note that if the GNS is a $(S,\textbf{f})$ Frobenius GNS then $$|C(S)|=|C(\textbf{f})|=(f^{(1)}+1)\cdots (f^{(d)}+1).$$  $|C(S)|$ is known as the {\it conductor} if $S$ is a  numerical semigroup.

\begin{lemma} Let $S\subseteq \mathbb{N}^{d}$ be a GNS of genus $g(S)$. Then
\begin{enumerate}
\item $|C(S)|=|H(S)|+|N(S)|$.
\item $(h^{(1)}+1)\cdot \ldots \cdot (h^{(d)}+1)\leq |C(S)|$ for every $\textbf{h}\in H(S)$.
\end{enumerate} \end{lemma}
\begin{proof}
Trivial.
\end{proof}

\begin{exa} \label{exaf} \rm
In Figure~\ref{fig:GNS}), we consider the generalized numerical semigroup $S=\mathbb{N}^{2}\setminus\{(0,1),(1,0),(1,1),(1,2),(1,3),(1,4),(2,1),(3,0),(3,2)\}$. The minimal system of generators of $S$ is the set $G(S)=\{(2,0),(5,0),(0,2),(0,3),(1,5),(1,6),(3,1),(4,1)\}$.  The holes of $S$, marked black in Figure~\ref{fig:GNS}), determine the red region $C(S)$ in the figure. The elements of $S$ lying in the red region are the elements of $N(S)=\{(0,0),(2,0),(3,1),(0,2),(2,2),(0,3),(0,4)\}$ (marked by a circle). We have $n(S)=|N(S)|=7$, $e(S)=|G(S)|=8$, $C(S)$ is the disjoint union of $H(S)$ and $N(S)$, and $c(S)=|C(S)|=16$. Note that for any $\textbf{h}\in H(S)$ we have $(h^{(1)}+1)(h^{(2)}+1)\leq 16$.
\begin{figure}[htp]

\begin{tikzpicture} 
\draw [help lines] (0,0) grid (5,5);
\draw [<->] (0,6) node [left] {$y$} -- (0,0)
-- (6,0) node [below] {$x$};
\foreach \i in {1,...,5}
\draw (\i,1mm) -- (\i,-1mm) node [below] {$\i$} 
(1mm,\i) -- (-1mm,\i) node [left] {$\i$}; 
\node [below left] at (0,0) {$O$};

\draw (0,0) circle (3pt);
\draw [mark=*] plot (1,0);
\draw (2,0) circle (3pt);
\draw [mark=*] plot (3,0);
\draw [mark=*] plot (0,1);
\draw [mark=*] plot (1,1);
\draw [mark=*] plot (2,1);
\draw (3,1) circle (3pt);
\draw (0,2) circle (3pt);
\draw [mark=*] plot (1,2);
\draw (2,2) circle (3pt);
\draw [mark=*] plot (3,2);
\draw (0,3) circle (3pt);
\draw [mark=*] plot (1,3);
\draw (0,4) circle (3pt);
\draw [mark=*] plot (1,4);

\draw [fill=red, opacity=0.2] (0,0) rectangle (1,4);
\draw [fill=red, opacity=0.2] (1,0) rectangle (3,2);
\end{tikzpicture}
\caption{The generalized numerical semigroup in Example~\ref{exaf}.}\label{fig:GNS}
\end{figure}
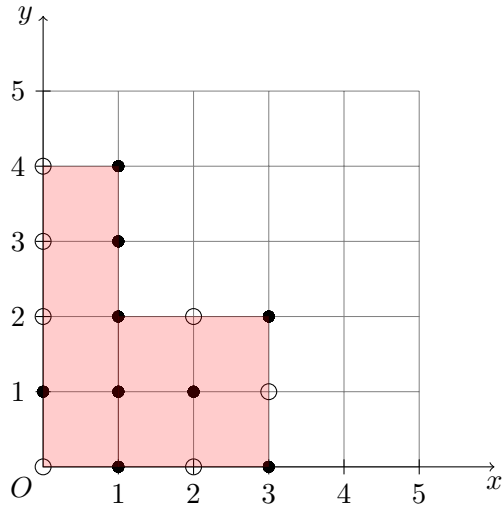
\end{exa}
If we let $c(S)$, for generalized numerical semigroups, play the role of $F(S)+1$ in numerical semigroups, we can extend Conjecture~\ref{wilf-frof} to arbitrary generalized numerical semigroups as follows.

\begin{thm2} \label{Wilf-gen} If $S\subseteq \mathbb{N}^{d}$ then 
	$e(S)n(S)\geq d\,c(S)$
\end{thm2}

\begin{rem}
The Generalized Wilf Conjecture can also be stated for the class of $\mathcal{C}$-semigroups considered in~\cite{G.G.2016}.  For readability (and to aid intuition) we save discussion of this for a future paper.
\end{rem}

\section{Multiplicity and thickenings}
A crucial concept for numerical semigroups is \textit{multiplicity}.  There is a natural way to define this notion for generalized numerical semigroups which we will use to verify \genwilf{} for large classes of generalized numerical semigroups.

\begin{defn}\label{def:ModuleGens}
Let $S\subset\mathbb{N}^d$ be a generalized numerical semigroup.  Let $M(S)^*=\{\textbf{h}\in H(S) \mid C(\textbf{h}) \cap S =\{\textbf {0}\}\}$.  Equivalently, $M(S)^*$ consists of all non-zero $\mathbf{x}\in\mathbb{N}^d$ satisfying that $\mathbf{0}$ is the only element of $S$ less than or equal to $\mathbf{x}$ in the natural partial order on $\mathbb{N}^d$.  Following~\cite{HTY09}, we call the elements of $M(S)^*$ the \textit{fundamental holes} of $S$. Let $M(S)=M(S)^* \cup \{\textbf{0}\}$. The \textit{multiplicity} of $S$ is defined as $m(S)=|M(S)|$.  
\end{defn}

\begin{lemma}\label{lem:ModuleGens}
The set $M(S)$ is the minimal subset of $\mathbb{N}^d$ satisfying that every $\textbf{x}\in\mathbb{N}^d$ can be written as $\textbf{x}=\textbf{m}+\textbf{s}$, where $\textbf{m}\in M(S)$ and $\textbf{s}\in S$.
\end{lemma}
\begin{proof}
We first prove that every $\textbf{x}\in\mathbb{N}^d$ can be written as $\textbf{x}=\textbf{m}+\textbf{s},$ where $\textbf{m}\in M(S)$ and $\textbf{s}\in S$.  Suppose $\textbf{x}\in \mathbb{N}^d$.  Let $\textbf{s}$ be a maximal element of $S$ (under the natural partial order on $\mathbb{N}^d$) so that $\textbf{s}\le \textbf{x}$.  Write $\textbf{x}=\textbf{s}+(\textbf{x}-\textbf{s})$; clearly $\textbf{x}-\textbf{s}\in\mathbb{N}^d$ since $\textbf{s}\le \textbf{x}$.  We prove that $\textbf{x}-\textbf{s}\in M(S)$.  If $\textbf{x}-\textbf{s}\notin M(S)$ then there is some $\textbf{s}'\in S$ with $\textbf{s}'\neq \textbf{0}$ so that $\textbf{s}'\le \textbf{x}-\textbf{s}$.  But then $\textbf{s}<\textbf{s}+\textbf{s}'\le \textbf{x}$, contradicting how $\textbf{s}$ was chosen.  So $\textbf{x}-\textbf{s}\in M(S)$ and $\textbf{x}=\textbf{s}+(\textbf{x}-\textbf{s})$ gives a decomposition of the desired form.

Now suppose that $T\subset\mathbb{N}^d$ satisfies that every $\textbf{x}\in\mathbb{N}^d$ can be written as $\textbf{x}=\textbf{s}+\textbf{t}$ for some $\textbf{s}\in S$ and $\textbf{t}\in T$.  Suppose that $\textbf{m}\in M(S)$.  Then $\textbf{m}=\textbf{s}+\textbf{t}$ for some $\textbf{s}\in S,\textbf{t}\in T$.  Since the only element of $S$ less than $\textbf{m}$ is $\textbf{0}$, we have $\textbf{m}=\textbf{0}+\textbf{t}=\textbf{t}$.  Thus $M(S)\subseteq T$.
\end{proof}

\begin{rem}
Even when $S$ has infinitely many holes, the set of fundamental holes is finite (see~\cite{HTY09}).  From an algebraic perspective this is explained by the fact that the integral closure of a ring is module finite over the ring.
\end{rem}


\begin{lemma}\label{lem:MultiplicityInequality}
Suppose $S\subset\N^d$ is a generalized numerical semigroup.  Then $c(S)\le m(S)n(S)$.
\end{lemma}
\begin{proof}
Let $\prec$ be any total order on $\N^d$ which refines the natural partial order.  Define a map $\psi_{\prec}:\N^d\to M(S)\times S$ as follows: for $\textbf{x}\in\N^d$, select $\textbf{s}=\max_{\prec}\{\textbf{t}\in S\mid \textbf{t}\leq \textbf{x}\}$.  As in the proof of Lemma~\ref{lem:ModuleGens}, $\textbf{x}-\textbf{s}\in M(S)$.  Since $\prec$ is a total order, the decomposition $\textbf{x}=\textbf{s}+(\textbf{x}-\textbf{s})$ chosen in this way is unique.  Define $\psi_{\prec}(\textbf{x})=(\textbf{x}-\textbf{s},\textbf{s})$.  Clearly this is a well-defined injection.

Now restrict $\psi_{\prec}$ to the subset $C(S)=\{\textbf{x}\in \N^d: \textbf{x}\le \textbf{h}\mbox{ for some } \textbf{h}\in H(S) \}$.  The largest $\textbf{s}\in S$ so that $\textbf{s}\preceq \textbf{x}$ also is in $C(S)$.  This gives an injection $\psi_{\prec}: C(S)\to M(S)\times S\cap C(S)$.  We observe that $c(S)=|C(S)|$, $m(S)=|M(S)|$, and $|S\cap C(S)|=n(S)$, which concludes the proof.
\end{proof}

\begin{defn}\label{def:minimalmultiplicity}
We say a generalized numerical semigroup $S$ has \textit{minimal multiplicity} if $c(S)=m(S)n(S)$.  In this case every $\textbf{x}\in C(S)$ can be written uniquely as $\textbf{x}=\textbf{m}+\textbf{s}$, where $\textbf{m}\in M(S)$ and $\textbf{s}\in S$.
\end{defn}

In the following, $S^*$ denotes $S\setminus \textbf{0}$.

\begin{defn}
Write $\be_1,\ldots,\be_{d+1}$ for the standard semigroup generators of $\N^{d+1}$ and consider the semigroup isomorphic to $\N^d$ inside $\N^{d+1}$ generated by $\{\be_1,\ldots,\be_{d+1}\}\setminus \be_i$ for some $i$.  By abuse of notation we refer to the latter semigroup as $\N^d$.  Suppose $S\subset\N^d$ is a semigroup.  The $k$-thickening of $S\subset\N^d$ along axis $i$ in $\N^{d+1}$ is the semigroup $\kthick{k}(S,i)\subset\N^{d+1}$ defined as
\[
\kthick{k}(S,i)=S\cup (\be_i+S)\cup\cdots\cup (k\be_i+S) \cup ((k+1)\be_i+\N^{d+1}).
\]
\end{defn}

\begin{rem}
If $S\subseteq\mathbb{N}^{d}$ is a GNS then $\kthick{0}(S,i)$ corresponds to embedding $S$ in $\mathbb{N}^{d+1}$ in the coordinate hyperplane $x_i=0$.
\end{rem}

\begin{prop}\label{prop:MingensOfThickenings}
Suppose $S\subset\N^d$ is a semigroup with minimal generating set $G(S)$, $S$-module generators $M(S)$, and multiplicity $m(S)$.  The minimal generating set of $\kthick{k}(S,i)\subset\N^{d+1}$ is 
\[
\{\be_i\}\cup G(S)\cup ((k+1)\be_i+M(S)^*),
\]
where $M(S)^*$ is the set of fundamental holes of $S$.
\end{prop}
\begin{proof}
We first show that any $\textbf{x}\in\kthick{k}(S,i)$ can be written in terms of the prescribed generators.  First, if $\textbf{x}\in j\be_i+S$ for some $0\le j\le k$, then clearly $\textbf{x}$ is the sum of $j\be_i$ and some number of elements of $G(S)$.  Now suppose $\textbf{x}=(k+1)\be_i+\textbf{n}$ for some $\textbf{n}\in\N^{d+1}$.  By Lemma~\ref{lem:ModuleGens}, there is some $\textbf{m}\in M(S)$, $\textbf{s}\in S$ so that $\textbf{n}=\textbf{m}+\textbf{s}$.  Hence $\textbf{x}$ can be written as a sum of $(k+1)\be_i, \textbf{m},$ and some number of generators of $S$.  For minimality, clearly $\be_i$ and $G(S)$ cannot be removed from the generating set.  If any element of $(k+1)\be_i+M(S)^*$ is removed from the generating set, then Lemma~\ref{lem:ModuleGens} guarantees that all of $(k+1)\be_i+\N^{d+1}$ will not be generated.
\end{proof}

\begin{coro}\label{coro:thickstats}
Suppose $S\subset\N^d$ is a semigroup.  Then $e(\kthick{k}(S,i))=e(S)+m(S)$, $n(\kthick{k}(S,i))=(k+1)n(S)$, and $c(\kthick{k}(S,i))=(k+1)c(S)$.
\end{coro}

\begin{prop}\label{prop:ThickeningAndWilf}
If $S$ satisfies \genwilf{} ($dc(S)\le n(S)e(S)$) then so does $\kthick{k}(S,i)$.  Moreover, if $S$ has minimal multiplicity and satisfies $dc(S)=n(S)e(S)$, then $(d+1)c(\kthick{k}(S,i))=n(\kthick{k}(S,i))e(\kthick{k}(S,i))$.
\end{prop}
\begin{proof}
By Corollary~\ref{coro:thickstats}, $(d+1)c(\kthick{k}(S,i))=(k+1)(dc(S)+c(S))$ and $n(\kthick{k}(S,i))e(\kthick{k}(S,i))=(k+1)n(S)(e(S)+m(S))$.  Thus it suffices to show that $dc(S)+c(S)\le n(S)e(S)+n(S)m(S)$, with equality if $dc(S)=n(S)e(S)$ and $c(S)=n(S)m(S)$.  This follows from Lemma~\ref{lem:MultiplicityInequality} and our assumption that $S$ satisfies $dc(S)\le n(S)e(S)$.
\end{proof}

\begin{rem}
By Proposition~\ref{prop:ThickeningAndWilf}, it suffices to prove \genwilf{} for semigroups which are not of the form $\kthick{k}(S,i)$ for a semigroup $S$ of strictly smaller dimension.
\end{rem}

Thickening is a process that can be iterated any number of times.  We use the following notation.
\begin{defn}\label{def:IterateThick}
Let $S\subset \N^d$ be a GNS and suppose $\N^d$ is embedded in $\N^{d+t}=\mathrm{Span}_{\N}\{\be_1,\ldots,\be_{d+t}\}$ along the axes $\be_{i_1},\ldots,\be_{i_d}$ and put $\{\be_{j_1},\ldots,\be_{j_t}\}=\{\be_1,\ldots,\be_{d+t}\}\setminus\{\be_{i_1},\ldots,\be_{i_d}\}$.  Consider the iterative sequence of thickenings $S_1=\kthick{k_1}(S,j_1), S_2=\kthick{k_2}(S_1,j_2),\ldots, S_t=\kthick{k_t}(S_{t-1},j_t)$.  We write $\kthick{k_1,\ldots,k_t}(S,j_1,\ldots,j_t)$ for $S_t$.  If $k_1=\cdots=k_t=k$, then we simply write $\kthick{k}(S,j_1,\ldots,j_t)$ for $S_t$.
\end{defn}

\begin{rem}
In the sequence $S_1,\ldots,S_t$ constructed in Definition~\ref{def:IterateThick}, order does not matter.  Thus, once the axis directions $j_1,\ldots,j_t$ are chosen, there is a unique way to iteratively thicken $S$ along these axis directions.
\end{rem}

Applying Proposition~\ref{prop:ThickeningAndWilf} repeatedly, we see that if $S$ is a semigroup which satisfies \genwilf{}, then $\kthick{k_1,\ldots,k_t}(S,\be_{j_1},\ldots,\be_{j_t})$ also satisfies \genwilf{}.  We now consider a special case of iterative thickening; this is the case when $k=0$ for each step.

Let $A$ be a subset of $\mathbb{N}^{d}$, denote by $\mathrm{Span}_{\mathbb{R}}(A)$ the $\mathbb{R}$-vector subspace of $\mathbb{R}^{d}$ spanned by the elements of $A$. Recall that a vector subspace of $\mathbb{R}^{d}$ is a \emph{coordinate linear space} if it is spanned by a subset of the standard basis $\{\textbf{e}_{1},\textbf{e}_{2},\ldots,\textbf{e}_{d}\}$. The results in the second part of this section are inspired by the following proposition.

\begin{prop}[\cite{Prof}, Proposition 5.2]\label{prop:Coord} Let $S\subseteq \mathbb{N}^{d}$ be a GNS and $H(S)$ the set of its holes. Then $\mathrm{Span}_{\mathbb{R}}(H(S))$ is a coordinate linear space.\label{spaziocoord}\end{prop}

We will use the following notation:
\begin{itemize}
\item $S_{g,d}$ is the set of all GNS with genus $g$ in $\mathbb{N}^{d}$.
\item $S_{g,d}^{(r)}=\{S\in S_{g,d}\ |\ \dim(\mathrm{Span}_{\mathbb{R}}(H(S)))=r\}$.
\end{itemize}

%
%
%

\begin{defn}
Let $S\in S_{g,d}^{(r)}$.
\begin{enumerate}
\item Put $\Axes(S)=\{k\in \{1,2,\ldots,d\}\mid \mbox{for all\ } \textbf{h}\in H(S), h^{(k)}=0\}$, where $h^{(k)}$ is the $k$-th coordinate of $\textbf{h}\in \mathbb{N}^{d}$.

\item Set $\{i_{1},i_{2},\ldots,i_{r}\}=\{1,2,\ldots,d\}\setminus \Axes(S)$ and put $\overline{\be}_j=\be_{i_j}$ for $j=1,\ldots,r$.  By abuse of notation we write $\N^r$ for the sub-monoid of $\N^d$ generated by $\overline{\be}_1,\ldots,\overline{\be}_r$.
\item We define $\overline{S}=\N^r\cap S$.
\end{enumerate}
\label{definizione}
\end{defn}

\begin{lemma}
The semigroup $\overline{S}$ in Definition~\ref{definizione} is a generalized numerical semigroup of $\mathrm{Span}_{\mathbb{R}}(H(S))\cap\N^d\cong \N^r$.
\end{lemma}
\begin{proof}
Proposition~\ref{prop:Coord} shows that $\mathrm{Span}_{\mathbb{R}}(H(S))\cap\N^d\cong \N^r$ (by abuse of notation we refer to $\mathrm{Span}_{\mathbb{R}}(H(S))\cap\N^d$ as $\N^r$).  Since $|\N^d\setminus S|$ is finite, so is $|\N^r\setminus\overline{S}|$.  Hence $\overline{S}$ is a generalized numerical semigroup in $\N^r$.
\end{proof}

\begin{lemma}\label{lem:Iterative1Thick}
The following are equivalent:
\begin{enumerate}
\item $S\in S^{(r)}_{g,d}$
\item There is some $S'\in S^{(r)}_{g,r}\subset \N^r$ so that $S=\kthick{0}(S',\Axes(S))$.
\end{enumerate}
\end{lemma}
\begin{proof}
(1)$\Rightarrow$(2): Suppose that $\mathrm{Span}_{\mathbb{R}}(H(S))=\mathrm{Span}\{\be_{i_1},\ldots,\be_{i_r}\}$ and $\Axes(S)=\{j_1,\ldots,j_{d-r}\}$.  Then $\overline{S}\in S^{(r)}_{g,r}$ and $S=\overline{S}\cup (\be_{j_1}+\N^{r+1})\cup (\be_{j_2}+\N^{r+2})\cup \cdots\cup (\be_{j_{d-r}}+\N^d)=\kthick{0}(\overline{S},\Axes(S))$.

(2)$\Rightarrow$(1): If $S=\kthick{0}(S',\Axes(S))$ for some $S'\in S^{(r)}_{g,r}$, then $\dim(\mathrm{Span}_{\mathbb{R}}(H(S)))=\dim(\mathrm{Span}_{\mathbb{R}}(H(S')))=r$.  Since $0$-thickenings do not effect genus, $S\in S^{(r)}_{g,d}$.
\end{proof}

\begin{exa} \rm Let  $S=\mathbb{N}^{5}\setminus \{(0,0,0,1,0),(0,0,0,2,0),(0,1,0,0,0),$ $(0,1,0,3,0)\}$. The set of minimal generators of $S$ is $G(S)=\{(1,0,0,0,0),(0,0,1,0,0),(0,0,0,0,1),(1,0,0,1,0),(0,1,0,1,0),(0,0,1,1,0),\newline(0,0,0,1,1),(0,0,0,3,0),(1,0,0,2,0),(0,1,0,2,0),(0,0,1,2,0),(0,0,0,2,1),\newline(0,0,0,5,0),(0,0,0,4,0),(1,1,0,0,0),(0,1,1,0,0),(0,1,0,0,1),(0,2,0,0,0),\newline(0,2,0,1,0),(0,3,0,0,0)\}$.  Furthermore $e(S)=20$ and $g(S)=4$.  In this case $\Axes(S)=\{1,3,5\}$ and $i_{1}=2,i_{2}=4$. With the previous construction we have $\overline{S}=\mathbb{N}^{2}\setminus \{(0,1),(0,2),(1,0),(1,3)\}$. The set of minimal generators of $\overline{S}$ is $G(\overline{S})=\{(1,1),(0,3),(1,2),(0,5),(0,4),(2,1),(2,0),(3,0)\}$. So $e(\overline{S})=8$.
Notice that $M(\overline{S})=\{(0,0),(0,1),(0,2),(1,0)\}$ and $m(S)=4$.  If we iterate Proposition~\ref{prop:MingensOfThickenings} three times, we see how $G(S)$ is obtained from $G(\overline{S})$.
\end{exa}

\begin{coro}
Let $S\in S_{g,d}^{(r)}$ and suppose that $\overline{S}\in S_{g,r}^{(r)}$ satisfies \genwilf{}. Then $S$ satisfies \genwilf{}.  Moreover, if $\overline{S}$ has minimal multiplicity and satisfies \genwilf{} with equality, then so does $S$. 
\label{WilfProj}\end{coro}
\begin{proof}

This is immediate from Lemma~\ref{lem:Iterative1Thick} and Proposition~\ref{prop:ThickeningAndWilf}.\qedhere

 
\end{proof}

\section{\genwilf{} for Irreducible GNS}

Proposition~\ref{w-sym} shows that all symmetric generalized numerical semigroups satisfy \genwilf{}. Now we show that actually this occurs for all irreducible GNS. The proof of the conjecture for pseudo-symmetric GNS requires some preliminary results and Corollary~\ref{WilfProj}.

\begin{lemma} Let $S\subseteq \mathbb{N}^{d}$ be an irreducible GNS such that $e(S)=2d$. Then $S$ is symmetric. \label{e=2d} \end{lemma}
\begin{proof}
If $e(S)=2d$ then by \cite[Theorem 2.8]{Analele} it follows that $S=\langle A\rangle$ with $A=\{\textbf{e}_{1},\ldots,\textbf{e}_{i-1},\textbf{e}_{i+1},\ldots,\textbf{e}_{d}, a\textbf{e}_{i},b\textbf{e}_{i}\mid i\in \{1,\ldots,d\}, 1<a<b \in \mathbb{N}\setminus \{0\}, \mathrm{GCD}(a,b)=1\}\cup \{\textbf{e}_{i}+h^{(j)}\textbf{e}_{j}\mid j\in \{1,\ldots,d\}\setminus \{i\}, h^{(j)}\in \mathbb{N}\setminus \{0\}\}$. Observe that $a$ and $b$ generate a numerical semigroup in the $i$-th axis. We distinguish two cases:\\
1) $a=2$. In such a case $H(\langle 2,b\rangle)=\{1,3,5,\ldots,b-2\}$ and by a simple argument we see that $H(S)$ is the set:
\[
\left\lbrace h\textbf{e}_{i}+\sum_{j\neq i}l_{j}\textbf{e}_{j} \mid h\in H(\langle 2,b\rangle), l_{j}\in \{0,\ldots,h^{(j)}-1\}, j\in \{1,\ldots,d\}\setminus \{i\}\right\rbrace.
\] 
\noindent Moreover $S$ is a Frobenius GNS with Frobenius element $\textbf{f}=(b-2)\textbf{e}_{i}+\sum_{j\neq i}(h^{(j)}-1)\textbf{e}_{j}$ and genus $g(S)=\frac{b-1}{2}\prod_{j\neq i}h^{(j)}$. By Theorem~\ref{rel}, $S$ is symmetric.\\
2) $a> 2$. In such a case we show that $S$ is not a Frobenius GNS, so it is not irreducible. This will prove the claim of this lemma. Let $F=ab-a-b$ be the Frobenius number of $\langle a ,b\rangle$ and consider the element $\textbf{h}=F\textbf{e}_{i}+\sum_{j\neq i}(h^{(j)}-1)\textbf{e}_{j}$. We show that $\textbf{h}$ is a maximal element in $H(S)$ with respect to the natural partial order in $\mathbb{N}^{d}$. First we prove that $\textbf{h}\in H(S)$. If not, $\textbf{h}\in\langle A\rangle $ and since $\textbf{h}-(\textbf{e}_{i}+h^{(j)}\textbf{e}_{j})\notin \mathbb{N}^{d}$ for all $j\in \{1,\ldots,d\}\setminus \{i\}$, then $\textbf{h}=\lambda_{1}a\textbf{e}_{i}+\lambda_{2}b\textbf{e}_{i}+\sum_{j\neq i}\mu_{j}\textbf{e}_{j}$, with $\lambda_{1},\lambda_{2},\mu_{j}\in \mathbb{N}$. But this implies $F=\lambda_{1}a+\lambda_{2}b$ that is a contradiction. So $\textbf{h}\in H(S)$, in order to prove that it is a maximal hole it suffices to prove that $\textbf{h}+\textbf{e}_{k}\in S$ for all $k\in \{1,\ldots,d\}$. It is obvious that $\textbf{h}+\textbf{e}_{i}\in S$. So let $k\neq i$, then $\textbf{h}+\textbf{e}_{k}=(F-1)\textbf{e}_{i}+\sum_{j\neq i,k}(h^{(j)}-1)\textbf{e}_{k}+\textbf{e}_{i}+h^{(k)}\textbf{e}_{k}$. Since $\langle a,b\rangle$ is a symmetric numerical semigroup, $F-1\in \langle a,b \rangle$, hence $(F-1)\textbf{e}_{i}\in S$. Therefore $\textbf{h}+\textbf{e}_{k}\in S$ and $\textbf{h}$ is maximal in $H(S)$. It remains to prove that there exists an element in $H(S)$ not comparable with $\textbf{h}$. Consider $\textbf{x}=2\textbf{e}_{i}+h^{(k)}\textbf{e}_{k}$ with $k\neq i$. Obviously $\textbf{x}\nleq \textbf{h}$, moreover one can see by a simple argument that $\textbf{x}\in H(S)$. This concludes the proof.
\end{proof}

\begin{rmk} \rm
The proof of the previous Lemma shows actually a stronger result: If $S\subseteq \mathbb{N}^{d}$ is a GNS with $e(S)=2d$, then $S$ is Frobenius if and only if $S$ is symmetric.

\end{rmk}

For the claim and the proof of the following Lemma we use the same notation of the previous section.

\begin{lemma} Let $S\in S_{g,d}^{(r)}$ and $\overline{S}\in S_{g,r}^{(r)}$ be as in \emph{Definition~\ref{definizione}}. Then the following hold:
\begin{enumerate}
\item If $S$ is symmetric then $\overline{S}$ is symmetric.
\item If $S$ is pesudo-symmetric then $\overline{S}$ is pseudo-symmetric.
\end{enumerate}
\label{lemma1}\end{lemma}
\begin{proof}
Let $\{1,2,\ldots,d\}\setminus \Axes(S)=\{i_{1},i_{2},\ldots,i_{r}\}$. Suppose $S$ is symmetric or pseudo-symmetric and let $\textbf{f}=(f^{(1)},\ldots,f^{(d)})$ be the Frobenius element of $S$. Then $\prod_{i=1}^{d}(f^{(i)}+1)=\prod_{k=1}^{r}(f^{(i_{k})}+1)$ since for $j\in \Axes(S)$ we have $f^{(j)}+1=1$. But $\overline{\textbf{f}}=(f^{(i_{1})},\ldots,f^{(i_{r})})$ is the Frobenius element of $\overline{S}$. So both the statements follow easily from Theorem~\ref{irr_formula} and the fact that $g(S)=g(\overline{S})$.
\end{proof}

\begin{lemma} Let $S\subseteq \mathbb{N}^{d}$ be a GNS. Then the following hold:
\begin{enumerate}
\item If $g(S)<d$ then $S\in S_{g,d}^{(r)}$ for some $r<d$. In particular $g(S)\geq r$.
\item If $g(S)=d$ and $S\in S_{g,d}^{(d)}$ then $S$ is not pseudo-symmetric.
\end{enumerate}
\label{lemma2}\end{lemma}

\begin{proof}
The first statement is quite easy, considering that a vector space of dimension $r$ is spanned by exactly $r$ independent vectors. To prove the second statement, suppose that $(S,\mathbf{f})$ is a pseudo-symmetric GNS.  Then $\mathbf{f}/2,\mathbf{f}\in H(S)$, so $S$ must have at least $d+1$ holes to have $d$ linearly independent holes.  It follows that if $g(S)=d$ then $S$ is not pseudo-symmetric.\qedhere

\end{proof}

\begin{thm} Let $S\subseteq \mathbb{N}^{d}$ be a pseudo-symmetric GNS. Then $S$ satisfies \genwilf{}. \end{thm}
\begin{proof}
Let $g=g(S)$. We know that $S$ has Frobenius element $\textbf{f}=(f^{(1)},\ldots,f^{(d)})$ and, by Theorem~\ref{p-rel} (2), we have $2g-1=(f^{(1)}+1)\cdots (f^{(d)}+1)=c(S)$. So it suffices to prove that $e(S)n(S)\geq d(2g-1)$.  If $S$ is pseudo-symmetric then, by the map $\Psi_{\textbf{f}}$, $g-1=|H(\textbf{f})|-1=|N(\textbf{f})|=n(S)$. Furthermore $e(S)\geq 2d+1$ by Lemma~\ref{e=2d}. So $e(S)n(S)\geq (2d+1)(g-1)=d(2g-1)+g-(d+1)$, in particular if $g\geq d+1$ we conclude. Now consider that $S\in S_{g,d}^{(r)}$ with $r\leq d$. If $r=d$ we have $S\in S_{g,d}^{(d)}$, then by Lemma~\ref{lemma2} we have that $S$ is not pseudo-symmetric, a contradiction. If $r<d$ then we can consider $\overline{S}\in S_{g,r}^{(r)}$ and by Lemma~\ref{lemma1} it is pseudo symmetric. Moreover $g(\overline{S})\geq r$ hence by a similar argument we have that $\overline{S}$ satisfies \genwilf. By Corollary~\ref{WilfProj} the same holds for $S$.
\end{proof}

Combining the previous theorem with Proposition~\ref{w-sym} we can state the following general result:

\begin{thm} Let $S\subseteq \mathbb{N}^{d}$ be an irreducible GNS. Then $S$ satisfies \genwilf{}. \end{thm}

\section{Monomial Semigroups}
In this section we prove that generalized numerical semigroups satisfying $n(S)=1$ satisfy \genwilf{}.  We do this by exploiting a connection between generalized numerical semigroups with $n(S)=1$ and monomial ideals.  We assume some familiarity with commutative algebra.

\begin{defn}
Let $R=k[x_1,\ldots,x_d]$ and suppose $M$ is a graded $R$-module.  If $M$ has finite dimension as a $k$-vector space then we say that $M$ is \textit{zero-dimensional} and set $\ell(R/I)=\dim_k M$.  If $I$ is a homogeneous ideal of $R$ and $R/I$ is zero-dimensional, then we simply say $I$ is zero-dimensional.
\end{defn}

\begin{rem}
The disconnect between calling $M$ zero-dimensional while the dimension of $M$ as a $k$-vector space is positive is an unfortunate side effect of using \textit{dimension} to refer to both the Krull dimension of $M$ (which is zero) and the dimension of $M$ as a $k$-vector space (which is positive).
\end{rem}

Throughout this section, if $\pmb{\alpha}=(\alpha^{(1)},\ldots,\alpha^{(d)})\in\N^d$, then $x^{\pmb{\alpha}}$ means $x_1^{\alpha_1}x_2^{\alpha_2}\cdots x_d^{\alpha_d}$.

\begin{prop}\label{prop:MonomialTranslation}
Suppose $S\subset \N^d$ is a set containing $\mathbf{0}$, and $S^*=S\setminus\{\mathbf{0}\}$.  The following are equivalent.
\begin{enumerate}
\item $S$ is a generalized numerical semigroup with $n(S)=1$.
\item $S^*$ consists of the exponent vectors of the monomials in a zero-dimensional monomial ideal $I\subset k[x_1,\ldots,x_d]$ (where $k$ is any field).
\end{enumerate}
\end{prop}
\begin{proof}

(1)$\Rightarrow$(2): In the polynomial ring $k[x_1,\ldots,x_d]$, consider the ideal $I=\langle x^{\pmb{\alpha}}\mid \pmb{\alpha}\in S^*\rangle$.  We claim that if $x^{\pmb{\beta}}$ is a monomial and $x^{\pmb{\beta}}\in I$, then $\pmb{\beta}\in S^*$.  To see this, notice that $x^{\pmb{\beta}}\in I$ means $x^{\pmb{\alpha}}\mid x^{\pmb{\beta}}$ for some $\pmb{\alpha}\in S^*$.  In other words $\pmb{\alpha}\le \pmb{\beta}$ in the natural partial order.  If $\pmb{\beta}\notin S^*$ then $n(S)\ge 2$ since $\pmb{\alpha}$ would contribute to the set $N(S)=\{\textbf{s}\in S:\textbf{s}\le \textbf{h}\mbox{ for some } \textbf{h}\in H(S)\}$ whose cardinality is $n(S)$.  Hence $\pmb{\beta}\in S^*$.  It follows that $S^*$ is exactly the set of the exponent vectors of the monomials in the ideal $I$.  Since $\N^d\setminus S^*$ is finite, $I$ is zero-dimensional.\\
(2)$\Rightarrow$(1): Suppose $I$ is a zero-dimensional monomial ideal, and $S^*=\{\pmb{\alpha}\in\N^d\mid x^{\pmb{\alpha}}\in I\}$.  Since $I$ is an ideal, if $x^{\pmb{\alpha}},x^{\pmb{\beta}}\in I$, then $x^{\pmb{\alpha}+\pmb{\beta}}\in I$ and hence $\pmb{\alpha}+\pmb{\beta}\in S^*$.  Thus if we take $S=S^*\cup\{\mathbf{0}\}$, $S$ is a semigroup.  Moreover, if $x^{\pmb{\alpha}}\notin I$, then $x^{\pmb{\alpha}}$ is not divisible by any monomial $x^{\pmb{\beta}}\in I$, hence the set $\{\textbf{n}\in\N^d\mid \textbf{n}\le \pmb{\alpha}\}$ does not contain any elements of $S^*$.  It follows that $n(S)=1$.  Also, $\N^d\setminus S$ is finite since $I$ is zero-dimensional.
\end{proof}

In view of Proposition~\ref{prop:MonomialTranslation}, we make the following definition.

\begin{defn}
If $S$ is a generalized numerical semigroup satisfying $n(S)=1$ then we call $S$ a \textit{monomial semigroup} and we call the ideal $I=\langle x^{\pmb{\alpha}}\mid \pmb{\alpha}\in S^*\rangle$ the ideal corresponding to $S$.
\end{defn}

\begin{lemma}\label{lem:TranslatingStatistics}
If $S\subset\N^d$ is a monomial semigroup and $I\subset R=k[x_1,\ldots,x_d]$ is the ideal corresponding to $S$, then $e(S)=\ell(I/I^2)$ and $c(S)=\ell(R/I)$.
\end{lemma}
\begin{proof}
It is well-known that a minimal generating set for $S$ is provided by $S^*\setminus(S^*+S^*)$.  We can identify $S^*$ with the monomials in $I$ and $S^*+S^*$ with the monomials in $I^2$.  It follows that $S^*\setminus(S^*+S^*)$ can be identified with the monomials in $I$ but not in $I^2$.  These form a $k$-vector space basis for $I/I^2$.  Hence $e(S)=|S^*\setminus(S^*+S^*)|=\ell(I/I^2)$.

Recall $c(S)=|C(S)|$, and $C(S)=\{\textbf{n}\in\N^d \mid \textbf{n}\le \textbf{h}\mbox{ for some }\textbf{h}\in H(S)\}$.  Clearly the set $C(S)$ can be identified with monomials not in $I$.  These form a basis for $R/I$, hence $c(S)=\ell(R/I)$.
\end{proof}

\begin{exa}\label{ex:Symp} \rm
	Let $g=3$ and $d=2$. Consider the GNS $S=\mathbb{N}^{2}\setminus \{(0,1),(0,2),(0,3),(1,0),(1,1),(1,2),(2,0),(2,1),(3,0)\}$, which corresponds to the monomial ideal $I\subset k[x,y]$ generated by all monomials in two variables of degree at least four.  The minimal set of generators is $G(S)=\{(0,4),(0,5),(0,6),(0,7), (4,0),(5,0),(6,0),(7,0),(1,3),(1,4), (1,5), (1,6),\newline (2,2), (2,3),(2,4),(2,5),(3,1),(3,2),(3,3),(3,4),(4,1),(4,2),(4,3),(5,1),\newline (5,2),(6,1)\}$; these are the exponent vectors of monomials in $I$ but not in $I^2$.  Figure~\ref{fig:Symp} provides a graphical view of this GNS -- black points are the holes of the GNS, while the red points are the minimal generators.
	
	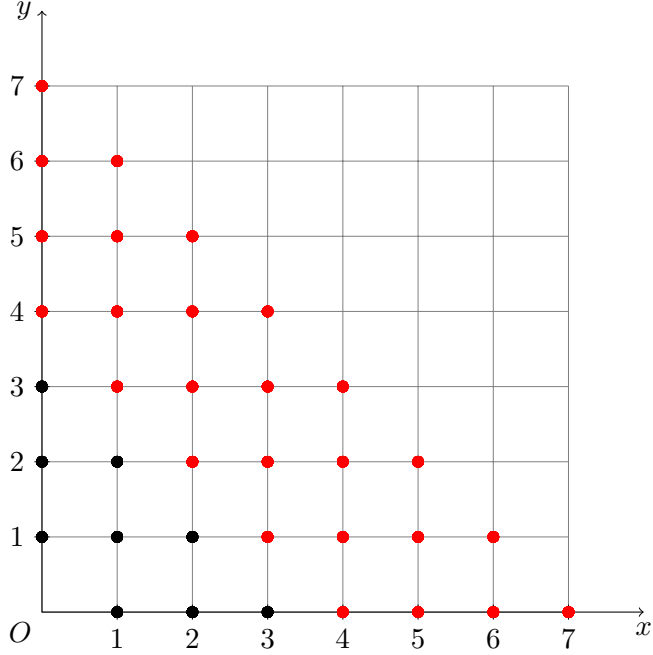
\begin{figure}[htp]
		\begin{center}
			\begin{tikzpicture} 
			\usetikzlibrary{patterns}
			\draw [help lines] (0,0) grid (7,7);
			\draw [<->] (0,8) node [left] {$y$} -- (0,0)
			-- (8,0) node [below] {$x$};
			\foreach \i in {1,...,7}
			\draw (\i,1mm) -- (\i,-1mm) node [below] {$\i$} 
			(1mm,\i) -- (-1mm,\i) node [left] {$\i$}; 
			\node [below left] at (0,0) {$O$};

			\draw [mark=*] plot (0,1);
			\draw [mark=*] plot (0,2);
			\draw [mark=*] plot (0,3); 
			\draw [mark=*] plot (1,0);
			\draw [mark=*] plot (1,1);
			\draw [mark=*] plot (1,2);
			\draw [mark=*] plot (2,0);
			\draw [mark=*] plot (2,1);
			\draw [mark=*] plot (3,0);
			
			\draw [red, mark=*] plot (4,0);
			\draw [red, mark=*] plot (5,0);
			\draw [red, mark=*] plot (6,0);
			\draw [red, mark=*] plot (7,0);
			
			\draw [red, mark=*] plot (0,4);
			\draw [red, mark=*] plot (0,5);
			\draw [red, mark=*] plot (0,6);
			\draw [red, mark=*] plot (0,7);
			\draw [red, mark=*] plot (1,3);
			\draw [red, mark=*] plot (1,4);
			\draw [red, mark=*] plot (1,5);
			\draw [red, mark=*] plot (1,6);
			\draw [red, mark=*] plot (2,2);
			\draw [red, mark=*] plot (2,3);
			\draw [red, mark=*] plot (2,4);
			\draw [red, mark=*] plot (2,5);
			\draw [red, mark=*] plot (3,1);
			\draw [red, mark=*] plot (3,2);
			\draw [red, mark=*] plot (3,3);
			\draw [red, mark=*] plot (3,4);
			\draw [red, mark=*] plot (4,1);
			\draw [red, mark=*] plot (4,2);
			\draw [red, mark=*] plot (4,3);
			\draw [red, mark=*] plot (5,1);
			\draw [red, mark=*] plot (5,2);
			\draw [red, mark=*] plot (6,1);

			\end{tikzpicture}
		\end{center}
		\caption{The semigroup in Example~\ref{ex:Symp}}\label{fig:Symp}
	\end{figure}
	
\end{exa}

\begin{thm}[The Generalized Wilf Conjecture for Monomial Semigroups]\label{thm:MonomialWilf}
If $S$ is a monomial semigroup, then $d\hskip 1pt c(S)\le e(S)$.  Equivalently (by Lemma~\ref{lem:TranslatingStatistics}), if $I\subset R=k[x_1,\ldots,x_d]$ is a zero-dimensional monomial ideal, then $d\hskip 1pt \ell(R/I)\le \ell(I/I^2)$.
\end{thm}

\begin{proof}
We proceed by induction on the ambient dimension $d$ and $\ell(R/I)=c(S)$.  If $d=1$, then $I=\langle x^k\rangle$ and $I^2=\langle x^{2k}\rangle$, so $d\ell(R/I)=k=\ell(I^2/I)$.  Now suppose $d>1$ and $R=k[x_1,\ldots,x_d]$.  If $I\subset R$ is a monomial ideal so that $\ell(R/I)=1$ then $I=\langle x_1,\ldots,x_{d}\rangle$ and $I/I^2\cong \mathrm{Span}_k\{x_1,\ldots,x_{d}\}$, so $\ell(I/I^2)=d=d\cdot \ell(R/I)$.  Now suppose that $I$ is a monomial ideal in $R$ and that the proposed inequality holds for all monomial ideals in less than $d$ variables and all monomial quotients of length less than $\ell(R/I)$.  For simplicity we write $y=x_d$ and put $\overline{R}=R/\langle y\rangle\cong k[x_1,\ldots,x_{d-1}]$, $\overline{I}=(I+\langle y\rangle)/\langle y\rangle \subset \overline{R}$.  We consider the short exact sequence
\begin{equation}\label{eq:1}
0\rightarrow\frac{R}{I:y}\xrightarrow{\cdot y} \frac{R}{I}\rightarrow \frac{R}{I+\langle y\rangle}\cong \frac{\overline{R}}{\overline{I}}\rightarrow 0.
\end{equation}
From~\eqref{eq:1} we get $\ell(R/I)=\ell(R/(I:y))+\ell(\overline{R}/\overline{I})$.  We always have the containment $I\subset I:y$.  Furthermore the containment is always proper since $I$ is zero-dimensional, so $\ell(R/(I:y))<\ell(R/I)$.  By the induction hypothesis, we thus have $d\ell(R/(I:y))\le\ell((I:y)/(I:y)^2)$.  Since $\overline{R}$ involves one less variable, we also have $(d-1)\ell(\overline{R}/\overline{I})\le\ell(\overline{I}/\overline{I}^2)$ by induction.  Hence
\[
d\ell(R/I)=d\ell(R/(I:y))+d\ell(\overline{R}/\overline{I})\le \ell((I:y)/(I:y)^2)+\ell(\overline{I}/\overline{I}^2)+\ell(\overline{R}/\overline{I}).
\]
To complete the induction, it suffices to prove
\begin{equation}\label{eq:2}
\ell((I:y)/(I:y)^2)+\ell(\overline{I}/\overline{I}^2)+\ell(\overline{R}/\overline{I})\le \ell(I/I^2).
\end{equation}
We now simplify~\eqref{eq:2}.  Using the identity $y(I:y)=I\cap \langle y\rangle$, there is a short exact sequence
\[
0\rightarrow \frac{I:y}{I^2:y}\xrightarrow{\cdot y} \frac{I}{I^2}\rightarrow \frac{I}{I^2+I\cap\langle y\rangle }\cong \frac{I+\langle y\rangle}{I^2+\langle y\rangle}\cong \frac{\overline{I}}{\overline{I}^2}\rightarrow 0,
\]
which yields $\ell(I/I^2)=\ell((I:y)/(I^2:y))+\ell(\overline{I}/\overline{I}^2)$.  Plugging this into the right hand side of~\eqref{eq:2} and simplifying, we rewrite~\eqref{eq:2} as:
\begin{equation}\label{eq:3}
\ell((I:y)/(I:y)^2)+\ell(\overline{R}/\overline{I})\le \ell((I:y)/(I^2:y)).
\end{equation}
Finally, notice that $(I^2:y)\subset (I:y)^2$.  From the short exact sequence
\[
0\rightarrow \frac{(I:y)^2}{I^2:y} \rightarrow \frac{I:y}{I^2:y} \rightarrow \frac{I:y}{(I:y)^2} \rightarrow 0
\]
we get $\ell((I:y)/(I^2:y))-\ell((I:y)/(I:y)^2)=\ell((I:y)^2/(I^2:y))$.  So we can rewrite~\eqref{eq:3}, hence also~\eqref{eq:2}, as:
\begin{equation}\label{eq:4}
\ell\left(\frac{\overline{R}}{\overline{I}}\right)\le \ell\left(\frac{(I:y)^2}{I^2:y}\right).
\end{equation}
We prove~\eqref{eq:4} in Lemma~\ref{lem:ColonInequality}, completing the induction and the proof.
\end{proof}

\begin{lemma}\label{lem:ColonInequality}
Suppose $I\subset R=k[x_1,\ldots,x_{d-1},y]$ is a zero-dimensional monomial ideal and put $\overline{R}=R/\langle y\rangle\cong k[x_1,\ldots,x_{d-1}]$ and $\overline{I}=(I+\langle y\rangle)/\langle y\rangle \subset \overline{R}$.  Then
\[
\ell\left(\frac{\overline{R}}{\overline{I}}\right)\le \ell\left(\frac{(I:y)^2}{I^2:y}\right).
\]
\end{lemma}
\begin{proof}
We first prove that $I^2:y=I(I:y)$.  If $m$ is a monomial in $I^2:y$ then $ym=fg$ for some monomials $f,g\in I$.  Hence $y\mid f$ or $y\mid g$; without loss assume $f=yf'$.  Then $m=f'g\in I(I:y)$.  Conversely if $m\in I(I:y)$ then $m=fg$ where $f\in I$ and $g\in I:y$.  Then $my=f(gy)\in I^2$.

Now we prove the inequality by producing an injective map $\phi$ from the canonical basis of $\overline{R}/\overline{I}$ into $(I:y)^2/(I^2:y)$.  The canonical basis for $\overline{R}/\overline{I}$ consists of monomials in the variables $x_1,\ldots,x_{d-1}$ which are not in $I$, and the canonical basis for $(I:y)^2/(I^2:y)$ consists of monomials in the variables $x_1,\ldots,x_{d-1},y$ which are in $(I:y)^2$ but not in $(I^2:y)$.

Let $m$ be a monomial in $\overline{R}/\overline{I}$, which we view as a monomial in $R$ without the $y$ variable.  Since $I^2:y$ has finite colength, $y^km\in (I^2:y)$ for all $k\gg 0$.  Let $t=\min\{k:y^km\in(I^2:y)\}$.  We claim that $y^{t-1}m\in (I:y)^2$, as follows.  Since $y^tm\in (I^2:y)=I(I:y)$, $y^tm=ab$ where $a\in I$ and $b\in (I:y)$.  Suppose $y\nmid a$.  Then $y^t\mid b$, so $b=y^tb'$ and $m=ab'\in I$, a contradiction since $m\notin I$.  Hence $y|a$ so $a=ya'$ for some $a'\in(I:y)$.  Thus $y^{t-1}m=a'b\in (I:y)^2$.

Now we define the map $\phi:\overline{R}/\overline{I}\to (I:y)^2/(I^2:y)$.  Given $m\in \overline{R}/\overline{I}$ (regarded as a monomial in $R$ not including the variable $y$), let $\phi(m)=y^{t-1}m$, where $t$ is the smallest integer such that $y^tm\in (I^2:y)$.  By the above argument, $\phi(m)\in (I:y)^2$ but not in $(I^2:y)$.  The map $\phi$ is clearly injective, so we are done.
\end{proof}

\begin{rem}
	There is a geometric interpretation of the inequality $d\ell(R/I)\le \ell(I/I^2)$ if $I$ is a zero-dimensional monomial ideal which is the initial ideal of a radical ideal.  Suppose that $p_1,\ldots,p_n\in\mathbb{A}^d$ are distinct points in affine $d$-space with corresponding ideals $\mathfrak{m}_1,\cdots,\mathfrak{m}_n$, and put $I=\mathfrak{m}_1\cap\cdots\cap\mathfrak{m}_n$.  We claim that $d\ell(R/I)=\ell(I/I^2)$.  Since $\ell(R/I)=n$ (the number of points), it suffices to show that $\ell(I/I^2)=nd$.  Since $I/I^2$ has finite length, we have
	\[
	I/I^2\cong \bigoplus_{i=1}^n (I/I^2)_{\mathfrak{m}_i},
	\]
	where $(I/I^2)_{\mathfrak{m}_i}$ is the localization at $\mathfrak{m}_i$.  Since localization is exact and $I_{\mathfrak{m}_i}\cong \mathfrak{m}_i$, $(I/I^2)_{\mathfrak{m}_i}\cong (\mathfrak{m}_i/\mathfrak{m}_i^2)$.  It is straightforward to show that $\ell(\mathfrak{m}_i/\mathfrak{m}_i^2)=d$ (it suffices to consider the case $\mathfrak{m}_i=\langle x_1,\ldots,x_d\rangle$), hence $\ell((I/I^2)_{\mathfrak{m}_i})=d$ for each summand above and $\ell(I/I^2)=nd$.  So $d\ell(R/I)=\ell(I/I^2)$ for zero-dimensional radical ideals.
	
	Now consider what happens under deformation to the initial ideal $J=\mbox{in}(I)$.  Since the deformation is flat, $\ell(R/I)=\ell(R/J)$.  However, $\mbox{in}(I^2)\supset J^2$, and these are not necessarily equal.  Thus $d\ell(R/J)=d\ell(R/I)=\ell(I/I^2)\le \ell(J/J^2)$, verifying Theorem~\ref{thm:MonomialWilf} if $J$ is the initial ideal of a zero-dimensional radical ideal.  Unfortunately, not all zero-dimensional monomial ideals are initial ideals of radical ideals (in technical terms, the Hilbert scheme of points is not necessarily \textit{smoothable}), so the above argument does not work for all monomial ideals.
\end{rem}

\begin{prop}\label{prop:ci}
If $I\subset k[x_1,\ldots,x_d]$ is a zero-dimensional monomial ideal, then $d\ell(R/I)=\ell(I/I^2)$ if and only if $I$ is a complete intersection.
\end{prop}
\begin{proof}
First, assume $I$ is a complete intersection, so $I=\langle x_1^{a_1},\ldots,x_d^{a_d}\rangle$.  The monomials in $I/I^2$ can be described as those which are divisible by precisely one of $x_i^{a_i}$.  Thus the monomials in $I/I^2$ are in bijection with the set $X=\{x_i^{a_i}m: m\notin I \mbox{ and }1\le i\le d\}$.  Clearly $|X|=d\ell(R/I)$, so we are done.

Now we prove by induction on $d$ and $\ell(R/I)$ that if $d\ell(R/I)=\ell(I/I^2)$ then $I$ is a complete intersection.  The cases $d=1$ and $\ell(R/I)=1$ are clear.  So suppose $d>1$ and $\ell(R/I)>1$.  Using the same notation as in the proof of Theorem~\ref{thm:MonomialWilf}, we have the short exact sequence
\[
0\rightarrow\frac{R}{I:y}\xrightarrow{\cdot y} \frac{R}{I}\rightarrow \frac{\overline{R}}{\overline{I}}\rightarrow 0.
\]
From the proof of Theorem~\ref{thm:MonomialWilf},
\[
\begin{array}{rl}
d\ell(R/I) & = d\ell(R/(I:y))+d\ell(\overline{R}/\overline{I})\\
& \le \ell((I:y)/(I:y)^2)+\ell(\overline{I}/\overline{I}^2)+\ell(\overline{R}/\overline{I})\\
& \le \ell(I/I^2),
\end{array}
\]
so if $d\ell(R/I)=\ell(I/I^2)$ then $d\ell(R/(I:y))=\ell((I:y)/(I:y)^2)$ and $(d-1)\ell(\overline{R}/\overline{I})=\ell(\overline{I}/\overline{I}^2)$.  By induction on $d$ and $\ell(R/I)$, both $I:y$ and $\overline{I}$ must be complete intersections.

Now we prove that if $I:y$ and $\overline{I}$ are both complete intersections and $I$ is \textit{not} a complete intersection, then $\ell\left(\frac{\overline{R}}{\overline{I}}\right)< \ell\left(\frac{(I:y)^2}{I^2:y}\right)$ and hence $d\ell(R/I)<\ell(I/I^2)$ (by the proof of Theorem~\ref{thm:MonomialWilf}).  Examining the proof of Lemma~\ref{lem:ColonInequality}, it suffices to prove that there is a monomial $m\in\overline{I}$ so that $m\in (I:y)^2$ but not in $I^2:y$.  This is what we show.

Since both $\overline{I}$ and $I:y$ are complete intersections, $I=\overline{I}+yM$ where $\overline{I}$ and $M$ are minimally generated as $\overline{I}=\langle x_1^{a_1},\ldots,x_{d-1}^{a_{d-1}}\rangle$ and $M=\langle x_1^{b_1},\ldots,x_{d-1}^{b_{d-1}},y^B\rangle$, and $a_1,\ldots,a_{d-1},b_1,\ldots,b_{d-1},B$ are all positive integers with the exception that we allow $b_i=-\infty$ with the convention that $x_i^{-\infty}=0$.  We stipulate that $b_i<a_i$ (otherwise $yx_i^{b_i}$ would be a redundant generator).  Since $I$ is not a complete intersection $b_i\ge 1$ for some $1\le i\le d-1$.  Thus $x_i^{2b_i}\in (I:y)^2$ but not in $I^2:y$.  If $x_i^{2b_i}\in\overline{I}$, then we are done.  Otherwise $2b_i<a_i$ and we claim that $x_i^{a_i}\in (I:y)^2$ but not in $I^2:y$.
\end{proof}

\begin{defn}\label{def:ordinary}
A generalized numerical semigroup $S\subset\N^d$ is an \textit{ordinary} GNS if there is some $\mathbf{f}\in\N^d$ so that $S=\{\mathbf{0}\}\cup (\N^d\setminus C(\mathbf{f}))$.
\end{defn}

\begin{rem}
In \cite{B.A.12} a numerical semigroup $S$ is called \emph{ordinary} if $S=\mathbb{N}\setminus\{1,2,\ldots,n\}$ for some $n\in \mathbb{N}$. The corresponding monomial ideal is $I=\langle x^{n+1}\rangle$.  Definition~\ref{def:ordinary} is a natural extension to generalized numerical semigroups of the notion of an ordinary numerical semigroup.
\end{rem}

\begin{prop}
The following conditions on a generalized numerical semigroup $S$ are equivalent:
\begin{enumerate}
\item $S$ is ordinary.
\item $S$ is a monomial semigroup and its corresponding ideal is a complete intersection.
\item $S$ satisfies $n(S)=1$ and $dc(S)=e(S)$.
\end{enumerate}
\end{prop}
\begin{proof}
This is immediate from Propositions~\ref{prop:MonomialTranslation} and~\ref{prop:ci}.
\end{proof}

The following example illustrates the equality $dc(S)=e(S)$ satisfied by ordinary generalized numerical semigroups.

\begin{exa}\label{ex:OrdinaryGNS1} \rm Let $\textbf{f}=(2,3)\in \mathbb{N}^{2}$. Then $C((2,3))= \{(0,0), (0,1), (0,2), \\(0,3), (1,0), (1,1), (1,2), (1,3), (2,0), (2,1), (2,2), (2,3)\}$.  The ordinary semigroup $S=\left(\mathbb{N}^{2}\setminus C(\textbf{f})\right)\cup \{(0,0)\}$ corresponds to the monomial complete intersection ideal $I=\langle x^3,y^4\rangle$.  Visually, the minimal generators of $S$ breaks into two copies of $C((2,3))$, namely $(3,0)+C((2,3))$ and $(0,4)+C((2,3))$.  These are displayed as the red dots in Figure~\ref{fig:OrdinaryGNS}.  (Notice that these red dots correspond to the exponent vectors of monomials in $I$ but not in $I^2$, where $I=\langle x^3,y^4\rangle$.)  The holes of $S$ are marked in black (they are all the elements of $C((2,3))$ except for $(0,0)$), while all points which are red or not marked belong to $S$.

\begin{figure}[h!]
\begin{center}
\begin{tikzpicture} 
\usetikzlibrary{patterns}
\draw [help lines] (0,0) grid (7,8);
\draw [<->] (0,9) node [left] {$y$} -- (0,0)
-- (8,0) node [below] {$x$};
\foreach \i in {1,...,7}
\draw (\i,1mm) -- (\i,-1mm) node [below] {$\i$} 
(1mm,\i) -- (-1mm,\i) node [left] {$\i$}; 
\node [below left] at (0,0) {$O$};

\draw [mark=*] plot (0,1);
\draw [mark=*] plot (0,2);
\draw [mark=*] plot (0,3); 
\draw [mark=*] plot (1,0);
\draw [mark=*] plot (1,1);
\draw [mark=*] plot (1,2);
\draw [mark=*] plot (1,3);
\draw [mark=*] plot (2,0);
\draw [mark=*] plot (2,1);
\draw [mark=*] plot (2,2);
\draw [mark=*] plot (2,3);

\draw [red, mark=*] plot (3,0);
\draw [red, mark=*] plot (4,0);
\draw [red, mark=*] plot (5,0);
\draw [red, mark=*] plot (3,1);
\draw [red, mark=*] plot (4,1);
\draw [red, mark=*] plot (5,1);
\draw [red, mark=*] plot (3,2);
\draw [red, mark=*] plot (4,2);
\draw [red, mark=*] plot (5,2);
\draw [red, mark=*] plot (3,3);
\draw [red, mark=*] plot (4,3);
\draw [red, mark=*] plot (5,3);

\draw [red, mark=*] plot (0,4);
\draw [red, mark=*] plot (0,5);
\draw [red, mark=*] plot (0,6);
\draw [red, mark=*] plot (0,7);
\draw [red, mark=*] plot (1,4);
\draw [red, mark=*] plot (1,5);
\draw [red, mark=*] plot (1,6);
\draw [red, mark=*] plot (1,7);
\draw [red, mark=*] plot (2,4);
\draw [red, mark=*] plot (2,5);
\draw [red, mark=*] plot (2,6);
\draw [red, mark=*]plot (2,7);
\draw [fill=red, opacity=0.2] (0,0) rectangle (2,3);
\end{tikzpicture}
\end{center}
\caption{The ordinary GNS in Example~\ref{ex:OrdinaryGNS1}.}\label{fig:OrdinaryGNS}
\end{figure}
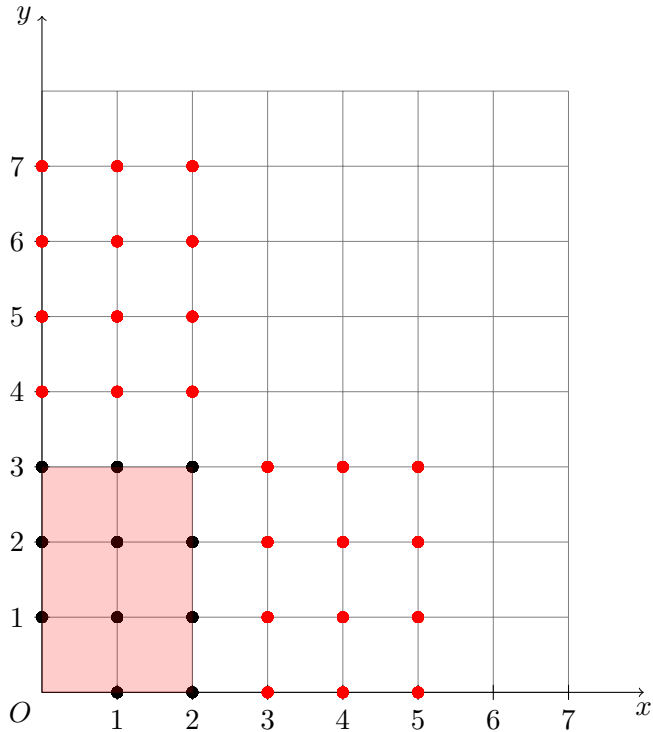
\end{exa}


\section{Comparison with a different extension of Wilf's conjecture}\label{sec:comparison}
In \cite{G.G.2016} another generalization of Wilf's conjecture is given. That generalization involves a larger class of affine semigroups, called $\mathcal{C}$-semigroups.  We will only consider the work of~\cite{G.G.2016} in the case of generalized numerical semigroups.  We will need some additional notation from~\cite{G.G.2016} (and also~\cite{Prof}).  Let $\prec$ be a monomial order satisfying that every monomial is preceded only by a finite number of monomials. The maximum of $H(S)$ with respect to $\prec$ is the Frobenius element of S, denoted by $Fb(S)$. By convention, $Fb(\mathbb{N}^{d})$ is the vector $(-1,\ldots, -1)$ with $d$ coordinates. Denote by $n_{\prec}(S)$ the cardinality of the finite set $\{\textbf{x}\in S\mid \textbf{x} \prec Fb(S)\}$. The Frobenius number of S is defined as $n_{\prec}(S) + g(S)$ and denoted by
$N(Fb(S))$.

\begin{conj}[Extended Wilf Conjecture, \cite{G.G.2016}, Conjecture 14] Let $S\subseteq \mathbb{N}^{d}$ be a GNS. Then $
n_{\prec}(S)e(S)\geq N(Fb(S)) + 1$, for every monomial order $\prec$ satisfying that every monomial is preceded only by a finite number of monomials. \label{Wilf-extend} \end{conj}

We would like to compare \genwilf{} and the Extended Wilf Conjecture~\ref{Wilf-extend}. First of all, we can remark that \genwilf{} does not depend on the choice of a monomial order. In order to provide a simple link between the two conjectures, we recall the following property, stated in \cite{Work1} in a more general case:

\begin{prop}[\cite{Work1}, Proposition 3.4] Every monomial order in $\mathbb{N}^{d}$ extends the natural partial order in $\mathbb{N}^{d}$.\label{relax}\end{prop}

\begin{prop} If $S\subseteq \mathbb{N}^{d}$ is a GNS that satisfies \genwilf{}~ then $S$ satisfies the Extended Wilf Conjecture~\ref{Wilf-extend}.\label{implication} \end{prop}

\begin{proof}
If $d=1$ it is clear that the two inequalities are the same, so we suppose that $d>1$ and assume that $S$ satisfies \genwilf{}. Fix a monomial order $\prec$ in $\mathbb{N}^{d}$. Let $\textbf{s}\in N(S)$, then $\textbf{s}\leq \textbf{h}$ for some $\textbf{h}\in H(S)$, with respect to the natural partial order in $\mathbb{N}^{d}$. By Proposition~\ref{relax}, $\textbf{s}\prec \textbf{h}\prec Fb(S)$, so $\textbf{s}\in \{\textbf{x}\in S\mid \textbf{x} \prec Fb(S)\}$. Therefore $n(S)\leq n_{\prec}(S)$. Consider \genwilf{} in the form $n(S)(e(S)-d)\geq dg(S)$.
Hence $n_{\prec}(S)(e(S)-1)\geq n(S)(e(S)-1)\geq n(S)(e(S)-d)\geq dg(S)\geq g(S)+1$, in particular $n_{\prec}(S)e(S)\geq n_{\prec}(S)+g(S)+1=N(Fb(S))+1$. 
\end{proof}

%

In \cite{G.G.2016} other classes of generalized numerical semigroups are given for which the Extended Wilf Conjecture \ref{Wilf-extend} is satisfied. Now we study the behaviour of those classes with respect to \genwilf{}. The first class (\cite{G.G.2016}, Lemma 15) provides another example, different from ordinary GNS, in which \genwilf{} holds as an equality:

\vspace{6pt}


\begin{prop} Let $h>1$ be a positive integer, $i\in\{1,2,\ldots,d\}$, $k\in \{1,2,\ldots,d\}\setminus \{i\}$. Consider the GNS $S\subseteq \mathbb{N}^{d}$ generated by:
\begin{multline*}
\{\textbf{e}_{1},\textbf{e}_{2},\ldots,\textbf{e}_{i-1},\textbf{e}_{i+1},\ldots,\textbf{e}_{d},2\textbf{e}_{i},3\textbf{e}_{i}\}\\
\cup \{\textbf{e}_{i}+h\textbf{e}_{k}\}\cup \{\textbf{e}_{i}+\textbf{e}_{j}\mid j\in\{1,2,\ldots,d\}\setminus\{k,i\}\}
\end{multline*}
Then $S$ satisfies \genwilf{}. \end{prop}

\begin{proof}
First we have $e(S)=2d$. The set of holes of $S$ is $H(S)=\{\textbf{e}_{i},\textbf{e}_{i}+\textbf{e}_{k},\textbf{e}_{i}+2\textbf{e}_{k},\ldots,\textbf{e}_{i}+(h-1)\textbf{e}_{k}\}$, so $S$ is a Frobenius GNS (\cite{Work1}) with Frobenius element $\textbf{f}=\textbf{e}_{i}+(h-1)\textbf{e}_{k}$. Therefore $c(S)=|C(\textbf{f})|=2h$. Furthermore $\bigcup_{\textbf{h}\in H(S)}N(\textbf{h})=\{\textbf{0},\textbf{e}_{k},2\textbf{e}_{k},\ldots,(h-1)\textbf{e}_{k}$, so $n(S)=h$. Finally $dc(S)=2dh=n(S)e(S)$.

\end{proof}

\begin{rem}
Observe that the GNS in the previous proposition can be expressed using thickenings. In particular if we consider the numerical semigroup $\langle2,3\rangle$ in the $i$-th axis and the GNS $\widehat{S}=\kthick{h-1}(\langle 2,3\rangle ,k)$, then $S=\kthick{0}(\widehat{S},\{1,\ldots,d\}\setminus\{i,k\})$. So $S$ satisfies the Generalized Wilf Conjecture also by Proposition~\ref{prop:ThickeningAndWilf} and the fact that $\langle 2,3\rangle$ satisfies Wilf's conjecture.

\end{rem}

\noindent 

The second class contains semigroups $S=\mathbb{N}^{d}\setminus \{\textbf{e}_{i},2\textbf{e}_{i},\ldots,(q-1)\textbf{e}_{i}\}$, with $i\in\{1,2,\ldots,d\}$ and $k\in \mathbb{N}\setminus\{0\}$ (\cite{G.G.2016}, Lemma 16). In this case $S=(\mathbb{N}^{d}\setminus C((q-1)\textbf{e}_{i})$ and $S$ satisfies \genwilf{} by Theorem~\ref{thm:MonomialWilf} or Proposition~\ref{prop:ThickeningAndWilf}.\\

\noindent For the third class (\cite{G.G.2016}, Lemma 17) we prove the following more general result

\begin{prop} Let $Q\subseteq \mathbb{N}$ be a numerical semigroup satisfying Wilf's conjecture, $j\in\{1,2,\ldots,d\}$ and a set $\{q_{i}\in \mathbb{N}\mid i\in \{1,2,\ldots,d\}\setminus \{j\}\}$. Then $S=\mathbb{N}^{d}\setminus \{(x_{1},\ldots,x_{d})\in \mathbb{N}^{d}\mid x_{j}\notin Q, x_{i}\leq q_{i}, i\in\{1,2,\ldots,d\}\setminus\{j\}\}$ is a GNS and it satisfies \genwilf{}.\end{prop}
\begin{proof}
Consider the numerical semigroup $Q$ on the axis $j$. Observe that $S$ is obtained by the following sequence of thickenings:
 $$S=\kthick{q_1,q_2,\ldots,\widehat{q_{j}},\ldots, q_d}(Q,\{1,\ldots,d\}\setminus \{j\}).$$ Then $S$ satisfies \genwilf{} by Proposition~\ref{prop:ThickeningAndWilf}.

\end{proof}

\section{Some computational tests}

The GAP\cite{GAP4} package \texttt{numericalsgps}\cite{GAP} offers tools to deal with numerical and affine semigroups. In \cite{algorithms}, in particular, some procedures for generalized numerical semigroups are described and these algorithms are implemented in the development version site of the package. Such tools allow to compute all generalized numerical semigroups of a given genus and to test \genwilf{} for them. Using this technique we verified that \genwilf{} is satisfied by all generalized numerical semigroups in $\mathbb{N}^{2}$ up to genus $g=13$, and in $\mathbb{N}^{3}$ up to genus $g=10$. Moreover the function \texttt{RandomAffineSemigroupWithGenusAndDimension}
allows to produce a random GNS in $\mathbb{N}^{d}$ of genus $g$, so it is possible to make a random test of \genwilf{}. Considering a random GNS of genus $g$, from $g=1$ up to $g=500$ we checked that different random tests give a positive answer for \genwilf{} in $\mathbb{N}^{d}$ from $d=2$ to $d=5$. We summarize the computational positive answers to \genwilf{} in the following table:

\vspace{6pt}
\begin{center}
\begin{tabular}{llclc}
\toprule
 & & genus & & Test \\
\midrule

$\mathbb{N}^{2}$ & & 1 to 13 & & All semigroups \\
  & & 1 to 500 & &  Random test\\
  
\midrule

$\mathbb{N}^{3}$ & & 1 to 10 & & All semigroups \\
  & & 1 to 500 & &  Random test\\
  
\midrule

$\mathbb{N}^{4}$ & & 1 to 500 & &  Random test\\

\midrule

$\mathbb{N}^{5}$ & & 1 to 500 & &  Random test\\

\bottomrule
\end{tabular}
\end{center}
\vspace{6pt}

\noindent Considering the number of such semigroups (see \cite{algorithms} and \cite{G.G.2016}) the previous test confirms a positive answer to \genwilf{} for a wide number of generalized numerical semigroups.

\section{Concluding Remarks}
If $S\subset\N^d$ is a GNS, it is natural to ask for a measure of the size of $n(S)e(S)-dc(S)$, which \genwilf{} postulates is non-negative.  Such a measure could expose additional terms either improving the inequality $n(S)e(S)\ge dc(S)$ or indicating where one could look for a counterexample.  We briefly consider this question for the case where $n(S)=1$, so $S^*$ is the set of exponent vectors of monomials inside a zero-dimensional monomial ideal $I\subset k[x_1,\ldots,x_d]$.  Let $a_1,\ldots,a_d$ be the smallest integers such that $x_i^{a_i}\in I$, and put $J=\langle x_1^{a_1},\ldots,x_d^{a_d}\rangle$.  Suppose that $I^2=JI$.  In the language of integral closures, this means that $J$ is a reduction of $I$ and the reduction number of $I$ with respect to $J$ is one (see~\cite[Chapter~8]{SH06}).  This is a very special situation - in general the ideal $J$ need not even be a reduction of $I$, let alone with reduction number one.  See for instance the recent preprint~\cite{HMRS19} which bounds the reduction number of monomial ideals in two variables which have $J=\langle x^a,y^b\rangle$ as a minimal reduction.

\begin{prop}\label{prop:ij}
Suppose $I\subset R=k[x_1,\ldots,x_d]$ is a zero-dimensional monomial ideal, $J= \langle x_1^{a_1},\ldots,x_d^{a_d}\rangle$ where $a_1,\ldots,a_d$ are the smallest integers such that $x_i^{a_i}\in I$, and $I^2=IJ$.  Then $\ell(I/I^2)-d\ell(R/I)=\prod_{i}a_i-\ell(R/I)=\ell(R/J)-\ell(R/I)$.
\end{prop}
\begin{proof}
Consider the set $X=\{x_i^{a_i}\cdot m: m\notin I\}$.  We first show that $X\subset I\setminus I^2$.  Suppose for a contradiction that $x_i^{a_i}m\in I^2$ where $m\notin I$.  Since $x_i^{a_i}m\in I^2$ and $I^2=IJ$, it follows that $x_i^{a_i}m=x_j^{a_j}n$ for some $1\le j\le d$ and $n\in I$.  As $m\notin I$, $x_j^{a_j}\nmid m$.  Thus the only way the equation can be satisfied is if $j=i$ and $m=n$, a contradiction since $m\notin I$.  So $X\subset I\setminus I^2$.  We can just as easily see that monomials of the form $x_i^{a_i}m$, $x_j^{a_j}n$ where $m,n\notin I$ must be distinct.  Hence $|X|=d\ell(R/I)$.  Now suppose $n$ is a monomial satisfying $n\notin X$ and $n\in I\setminus I^2$.  Then $n=\prod x_i^{b_i}$ where $0\le b_i\le a_i-1$.  It follows that the exponent vector of $n$ is in the box $B=[0,a_1-1]\times[0,a_2-1]\times\cdots\times[0,a_d-1]$.  By our assumption that $I^2=IJ$, \textit{every} monomial whose exponent vector is in $B$ is \textit{not} in $I^2$.  Notice also that any monomial not in $I$ has an exponent vector which is also in $B$.  Hence the number of monomials which are in $I\setminus I^2$ but not in $X$ is exactly $\prod_{i} a_i-\ell(R/I)$.  This completes the proof.
\end{proof}

\begin{exa}
Suppose $I=\langle x^5,x^3y^3,y^5\rangle\subset R=k[x,y]$ and $J=\langle x^5,y^5\rangle$.  Then $I^2=IJ$, so this is an example of the situation in Proposition~\ref{prop:ij}.  Here $\ell(I/I^2)=46$, $\ell(R/I)=21$, and $\ell(R/J)=25$.  We can check that $46=2\cdot 21+(25-21)$.  In contrast, $L=\langle x^5,xy^4,y^5\rangle$ does not satisfy $L^2=LJ$; the reduction number of $L$ with respect to $J$ is $4$ so we only have equality in the equation $L^{k+1}=JL^k$ when $k\ge 4$.
\end{exa}

Proposition~\ref{prop:ij} raises the question of whether \genwilf{} is just the largest term of an inequality incorporating other invariants of the semigroup $S$ and its integral closure.


\begin{thebibliography}{7}

\bibitem{B.A.12} Bras-Amor\'os, M.: The ordinarization transform of a numerical semigroup and semigroups with a large number of intervals. J. Pure Appl.Algebra \textbf{213}, 2507--2518 (2012).


\bibitem{algorithms} Cisto, C., Delgado, M., Garc\'\i a-S\'anchez, P.A., Algorithms for generalized numerical semigroups, \emph{arXiv preprint}
arXiv:1907.02461, (2019).

\bibitem{Analele} Cisto, C., Failla, G., Utano, R.: On the generators of a generalized numerical semigroup, \emph{Analele Univ. ``Ovidius''}, 27(1):49--59, (2019).

\bibitem{Work1} Cisto, C.,  Failla, G., Peterson, C., Utano, R.: Irreducible generalized numerical semigroups and uniqueness of the Frobenius element, Semigroup Forum (2019). https://doi.org/10.1007/s00233-019-10040-1, \emph{arXiv preprint} arXiv:1907.07955, (2019).


\bibitem{wilfSurvey} Delgado, M., Conjecture of Wilf: a survey, \emph{arXiv preprint} arXiv:1902.03461, (2019).

\bibitem{GAP} Delgado, M., Garc\'\i a-S\'anchez, P.A., Morais, J.: NumericalSgps, a package for numerical semigroups, Version 1.1.11.
https://gap-packages.github.io/numericalsgps, Mar 2019. Refereed GAP package.

\bibitem{D.M.} Dobbs, D., Matthews, G. M., On a question of Wilf concerning numerical semigroups,  in:  Focus on Commutative Rings Research,  Nova Sci. Publ.,  New York, 193--202, (2006)


\bibitem{Eliahou} Eliahou, S., Wilf's conjecture and Macaulay's theorem, \emph{J. Eur. Math. Soc.(JEMS)}, 20(9):2105--2129, (2018).


\bibitem{Prof} Failla, G., Peterson, C., Utano, R.: Algorithms and basic asymptotics for generalized numerical semigroups in $\mathbb{N}^{d}$, Semigroup Forum \textbf{92}(2), 460--473 (2016).

\bibitem{GAP4} GAP -- Groups, Algorithms, and Programming, Version 4.10.0. https://www.gap-system.org, Nov 2018.

\bibitem{G.G.2016} Garc\'\i a --Garc\'\i a, J.I., Mar\'\i n --Arag\'on, D., Vigneron--Tenorio, A., An extension of Wilf's conjecture to affine semigroups, Semigroup Forum   \textbf{96} (2), 396--408 (2018). 

\bibitem{HTY09} Hemmecke, R., Takemura, A., and Yoshida, R., Computing holes in semi-groups and its applications to
transportation problems, Contrib. Discrete Math. \textbf{4} (1), 81--91 (2009).

\bibitem{HMRS19} Herzog, J., Moradi, S., Rahimbeigi, M. and Soleyman Jahan, A., On the monomial reduction number of a monomial ideal in $K[x,y]$, \emph{arXiv preprint}
arXiv:1908.03765, (2019).

\bibitem{K.N.} Kaplan, N.: Counting numerical semigroups by genus and some cases of a question of Wilf. J. Pure Appl. Algebra \textbf{216}, 1016--1032 (2012).

\bibitem{M} Matthews, G., Weierstrass pairs and minimum distance of Goppa codes, Des. Codes Cryptogr. \textbf{22}, 107--121, (2001).


\bibitem{Garcia-Book} Rosales, J.C., Garc\'\i a-S\'anchez, P.A., Numerical semigroups, in: Developments in Mathematics, vol. 20, Springer, New York, (2009).


\bibitem{RGM1} Rosales, J.C., Garc\'\i a-S\'anchez, P.A., Garc\'\i a-Garc\'\i a, J.I., Jim\'enez Madrid, J.A.: The oversemigroups of a numerical semigroup, Semigroup Forum \textbf{67}(1), 145--158, (2003).

\bibitem{Sammartano} Sammartano, A., Numerical semigroups with large embedding dimension satisfy Wilf's conjecture, Semigroup Forum 85, 439--447 , (2012).

\bibitem{SH06} Swanson, I. and Huneke, C. Integral closure of ideals, rings, and modules, volume 336 of London Mathematical Society Lecture Note Series. Cambridge University Press, Cambridge, 2006.

\bibitem{W.} Wilf, H.S.: A circle-of-lights algorithm for the money-changing problem. Am. Math. Mon. \textbf{85}, 562--565 (1978).


\end{thebibliography}
\end{document}